\documentclass{article}

\usepackage[english]{babel}
\usepackage{amssymb}
\usepackage{amsmath}
\usepackage{amsthm}
\usepackage{color}

\newtheorem{theorem}{Theorem}[section]

\newtheorem{corollary}[theorem]{Corollary}
\newtheorem{remark}[theorem]{Remark}

\newtheorem{example}[theorem]{Example}

\title{\bf A Landau's theorem in several complex variables}
\author{Cinzia Bisi\footnote{Partially supported by GNSAGA of the INdAM, by FIRB ``Geometria Differenziale e teoria geometrica delle funzioni'' and by PRIN ``Variet\'a reali e complesse: geometria, topologia ed analisi armonica''.} 
\\
\normalsize Universit\`a degli Studi di Ferrara\\
\normalsize Dipartimento di Matematica e Informatica\\
\normalsize Via Machiavelli 35, 44121 Ferrara, Italy \\
\normalsize bsicnz@unife.it}

\begin{document}
\maketitle

\begin{abstract}
In one complex variable it is well known that if we consider the family of all holomorphic functions on the unit disc that fix the origin and with first derivative equal to 1 at the origin, then there exists a constant $\rho$, independent of the functions, such that in the image of the unit disc of any of the functions of the family there is a disc of universal radius $\rho.$ This is the so celebrated Landau's theorem. Many counterexamples to an analogous result in several complex variables exist. In this paper we introduce a class of holomorphic maps for which one can get a Landau's theorem and a Brody-Zalcman theorem in several complex variables.
\end{abstract}

\section{Introduction}
An important tool in geometric function theory in one complex variable is the so called Landau's theorem.\\
It says: {\it if we consider the family $\mathcal{F},$ of holomorphic functions in the unit disc $\Delta$ of $\mathbb{C}$ such that $f(0)=0,$ $f'(0)=1,$ then there is a constant $\rho>0,$ independent of $f,$ such that $f(\Delta)$ contains a disc of radius 
$\rho,$} \cite{Conw73},\cite{Heins62}, \cite{Land29I}, \cite{Land29II}.\\
By the {\it Landau's number} $l(f)$ of $f$ is meant the supremum of the set of positive numbers $r$ such that $f(\Delta)$ contains a disk of radius $r.$ By the {\it Landau's constant} $\rho$ we meant $\inf\limits_{f \in \mathcal{F}} l(f).$
\\
 An easy consequence of such result is the Picard theorem: {\it a non constant entire function $f: \mathbb{C} \to \mathbb{C}$ can omit at most one point. }\\
A related result is the compactness result of Brody-Zalcman: {\it if $\mathcal{H}$ is a non relative compact family of holomorphic maps in the unit disc, then after reparametrization we can get an extracted subsequence of functions in $\mathcal{H}$ converging to a non constant entire map}, \cite{BR78}, \cite{Zalc75}, \cite{Zalc98}. \\
It is well known that such principles break-down in several complex variables in a spectacular way. \\
There are holomorphic maps $F:\mathbb{C}^2 \to \mathbb{C}^2$ with Jacobian equal to $1,$ and whose image omit a non empty open set. These maps play an important role in holomorphic dynamics, see \cite{Sib99}. \\
To be more explicit, let $h(z,w)=(z^2 + aw, z),$ with $|a|<1,$ be a so called h\'enon map of $\mathbb{C}^2:$ it fixes $(0,0)$ and the attraction domain $\Omega$ of $(0,0)$ is a Fatou-Bieberbach domain isomorphic to $\mathbb{C}^2.$ Denote with $\Psi_h: \mathbb{C}^2 \to \Omega$ such isomorphism such that $\Psi ^{'}_h (0,0)=Id.$ Let $\overline{h}$ be the meromorphic extension of $h$ to $\mathbb{P}^2.$ If $(z,w)$ are the coordinates of $\mathbb{C}^2$  and $[z:w:t]$ are the coordinates of $\mathbb{P}^2,$ then the line at infinity has equation $\{ t=0 \}.$ We denote respectively with $I^+$ and $I^-$ the indeterminacy locus of $\overline{h}$ and $\overline{h^{-1}}$; they are two isolated points in $\{ t=0 \}.$  It is well known, see \cite{Sib99}, that the h\'enon map $h$ has $I^{-}=[1:0:0]$ and it is an attracting fixed point at the line at infinity.
Then $\Omega$ omits the basin of attraction of $I^-,$ which is an open set. Hence it cannot be that each $\Psi_{h,n} (z,w):=\frac{1}{n} \Psi_h (nz,nw)$ is such that the image of $B(0,1)$ contains a ball $B(0,r),$ because if this is the case, then $\Psi_h (nB(0,1)) \supset n B(0,r)$ and finally $\Psi_h (\mathbb{C}^2) \supset \mathbb{C}^2,$ which is not true because $\Psi_h$ omits an open set. \\
There are also some elder counterexamples; one by L.A. Harris, published in 1977, \cite{Har77}, that we recall for sake of completeness: given
$\delta >0,$ choose $n$ so that $n\delta^2>2$ and denote $g(z,w)=(z+nw^2,w).$
Suppose that the image of the open polydisc by $g$ contains a ball of radius $\delta$ and let
$(\alpha_0,\beta_0)$ be its center. Then given $|\zeta| < \delta$, there exist points $(z_0,w_0)$ and $(z_1,w_1)$ in the polydisc such that $g(z_0,w_0)=(\alpha_0,\beta_0)$ and $g(z_1,w_1)=(\alpha_0,\beta_0 + \zeta ).$
Hence $w_1 - w_0 = \zeta$, $w_1 + w_0= 2\beta_0 + \zeta$  and $n(w_1^2 - w_0^2)=z_0 - z_1$, so
$n|\zeta| |2\beta_0 + \zeta| \le 2,$ for all $|\zeta| < \delta.$ Thus $n\delta^2 \le  2,$ the desired contradiction. \\
In the same spirit, a second counterexample was given by P. Duren and W. Rudin in 1986, \cite{DuRu86}:
if $\delta >0,$ then the map $f(z,w)=(z, w + (\frac{z}{\delta})^2)$ is in the class of all biholomorphic maps from the unit polydisc into $\mathbb{C}^2$ which fix the origin and whose Jacobian matrix is the identity at the origin, but the image under $f$ of the polydisc contains no closed ball of radius $\delta.$ \\
Indeed, for no $(u,v) \in \mathbb{C}^2$ the image by $f$ of the polydisc $\Delta^2$ contains the circle:
$$
C=\{ (u+\delta e^{i\theta}, v) : -\pi \le \theta \le \pi \}.
$$   
To see this, fix $(u,v)\in \mathbb{C}^2.$ If $(u +\lambda,v) \in f(\Delta^2)$ then, by definition of $f,$ we have that:
$$
|v-\delta^{-2}(u+\lambda)^2|<1.
$$
Therefore, if all points of $C$ were in $f(\Delta^2),$ the inequality:
$$
|(\delta^2v-u^2)-2u\delta e^{i\theta} - \delta^2 e^{2i\theta}| <\delta^2
$$
would hold for all $\theta;$ Parseval's equality shows that this is impossible. \\ 
Even if several weak versions of the Landau's theorem in several complex variables have already been given, see for instance \cite{ChGa01}, \cite{GrKo03}, \cite{GrVr96}, \cite{FiGo94}, \cite{Liu92}, \cite{Tak51}, \cite{Wu67}, 
nevertheless the author believes that this paper can add something to the already existing literature: indeed
the purpose of this note is to introduce a class of holomorphic maps for which one can get a Landau's theorem and a Brody-Zalcman theorem in more than one variable and to underline the connection among the two.\\
Recently the Landau's theorem has also bring the attention of people working over the quaternion variable, see for instance \cite{BS16}, \cite{BS12}, \cite{BG09}, \cite{BG11}, \cite{BS13}.

\section{A Theorem of Brody-Zalcman type}
Let $\Phi: \mathbb{B}^k \to \mathbb{C}^k$ where $\mathbb{B}^k$ is the unit ball of $\mathbb{C}^k,$ $k>1.$ Let $\Phi'(a)$ denote the Jacobian matrix of $\Phi$ computed in the point $a.$
\begin{theorem}[Brody-Zalcman type Theorem] \label{BZT}
Let $C$ be a positive constant.
Consider a family of such holomorphic maps $\Phi$ satisfying
\begin{equation}\label{DiffCond}
||\Phi^{'} (a) ||\cdot ||\Phi^{'} (a)^{-1} || \le C,   \,\,\,\,\, \forall \,\, a \in \mathbb{B}^k .
\end{equation}
If the family is not normal then, after reparametrization, we can extract a subsequence converging to a non degenerate holomorphic map $\Psi : \mathbb{C}^k \to \mathbb{C}^k.$
\end{theorem}
\begin{proof}
Let $\Phi_n$ be an arbitrary sequence of maps of the family.
We can assume that the maps $\Phi_n$ are defined in a neighborhood of $\overline{\mathbb{B}^k}.$ \\
Define 
$$\lambda_n := \sup\limits_{|z|<1} (1-|z|) ||\Phi_n^{'} (z)||.$$
We can also assume that $\lambda_n \to + \infty,$ because if not the family is normal, \cite{Schiff93}.
Let $a_n$ be such that $(1-|a_n|)(||\Phi^{'}_n (a_n)||)=\lambda_n.$ \\
Define $$B_n := [(\Phi_n^{'})(a_n)]^{-1}$$ and $$\Psi_n(z) := \Phi_n (a_n + B_n z).$$ \\
We are going to show that $\Psi '_n (z)$ is well defined in $|z| \le \frac{\lambda_n}{2C}.$ \\
Indeed $\Psi_n^{'} (z) =\Phi_n^{'} (a_n + B_n z) \circ B_n$ with $\Psi_n^{'} (0)=Id$ and since
$$(1-|a_n + B_n z|)||\Phi_n^{'} (a_n + B_n z)|| \le \lambda_n$$  \\
we have:
\begin{equation} \label{123}
||\Psi_n^{'} (z)|| \le ||\Phi_n^{'} (a_n + B_n z)|| \cdot ||B_n|| \le \frac{\lambda_n ||B_n||}{1-|a_n|-|B_nz|}.
\end{equation}
If $|z| \le \frac{\lambda_n}{2C},$ then by \eqref{DiffCond}:
\begin{equation}\label{456}
 |B_n z| \le \frac{C|z|}{||\Phi^{'}(a_n)||} \le \frac{(1-|a_n|)}{\lambda_n} C \frac{\lambda_n}{2C}= \frac{1-|a_n|}{2}.
\end{equation}
So 
$$
||\Psi_n^{'} (z)|| \le \frac{2 \lambda_n}{1-|a_n|} ||B_n|| \le \frac{2\lambda_n}{1-|a_n|} \frac{C}{||\Phi_n^{'} (a_n)||}\le 2C
$$
So the family $\Psi_n$ is locally normal in $|z| \le \frac{\lambda_n}{C} \to \infty.$ \\
We get that $\Psi_n$ tends to a holomorphic map $\mathbb{C}^k \to \mathbb{C}^k.$ Moreover $\Psi_n^{'} (0)=Id$ so $\Psi^{'} (0)= Id.$
Hence $\Psi$ is non degenerate.
\end{proof}
\begin{remark}
If $k=1,$ then \eqref{DiffCond} is automatically satisfied.
\end{remark}
\begin{remark} Condition \eqref{DiffCond} implies that all the eigenvalues are comparable, i.e. $|\lambda_{\textrm{max}}(\Phi^{'} (a))| \le C |\lambda_{\textrm{min}} (\Phi^{'} (a))|$ for all $a \in \mathbb{B}^k,$ where $|\lambda_{\textrm{min}} (\Phi^{'} (a))|$ and $|\lambda_{\textrm{max}} (\Phi^{'} (a))|$ are respectively the minimal and the maximal modulus of the eigenvalues of the Jacobian matrix $\Phi^{'} (a).$
\end{remark}
\begin{remark}
Furthermore it is enough to assume \eqref{DiffCond} out of an analytic set, so the maps $\Phi$ don't need to be locally invertible. Indeed if we suppose that \eqref{DiffCond} holds out of an analytic set $A_{\Phi},$ we can choose $a_n \notin A_{\Phi_n}.$ 
\end{remark}
\begin{remark}
The condition \eqref{DiffCond} can be refined in 
$$
\sup\limits_{|z| \le \frac{1-|a|}{2}} ||\Phi^{'} (a+z) \cdot \Phi^{'} (a)^{-1}|| \le C
$$
This condition is less strong of \eqref{DiffCond} and implies that $||\Psi_n^{'}||$ is bounded for $|z| \le d_n$ with $d_n \to + \infty.$
\end{remark}
\begin{remark}
The Brody-Zalcman renormalization theorem \ref{BZT}, works also for families of maps from $\mathbb{B}^k$ to a compact hermitian manifold $M$ of dimension $k,$ \cite{FS00}.
\end{remark}
We also point out that, in \cite{Min82}, several sufficient conditions for a family of quasiregular mappings to be normal were previously given. 

\section{Landau's Theorem}
Assume that $\Phi \colon \mathbb{B}^k \to \mathbb{C}^k,$ $\Phi (0)=0,$ and $\Phi^{'} (0)= Id.$ \\
\begin{theorem}[Landau's Theorem]\label{MainTheo}
Let $C$ be a positive constant. Assume
\begin{equation}\label{DiffCond2}
||\Phi ^{'} (z)|| \cdot ||\Phi^{'} (z)^{-1}|| \le C, \,\, \forall \,\, z \in \mathbb{B}^k, 
\end{equation}
then there exists $\rho >0,$ depending only on $C,$ such that $\Phi(\mathbb{B}^k)$ contains a ball of radius $\rho >0.$
There exists also a domain $U$ such that $\Phi(U)=B(a, \rho).$
\end{theorem}
\begin{proof}
Suppose, by contradiction, that this is not the case. Then there exists a sequence $\Phi_n$ such that $\Phi_n (0)=0,$ $\Phi_n ^{'} (0)=Id$
satisfying \eqref{DiffCond2} and not the conclusion of the theorem. \\
If $(1-|z|)||\Phi_n^{'} (z)|| \le A,$ with $A$ constant, then $\Phi_n$ is normal and we can assume $\Phi_n \to \Psi,$ with $\Psi (0)=0$ and $\Psi^{'} (0)=Id.$ \\
On an appropriate sphere $\partial B (0,r),$ we get $|\Phi_n -\Psi | < |\Psi|.$ 
So by Rouche's theorem and by the fact that the image of $\Psi$ contains a ball of radius $R$, we get that eventually also the images of $\Phi_n$ contain a ball of radius smaller  or equal than $R$ because we are interested in the inequality 
$$|\Phi_n - b - \Psi +b| < |\Psi - b| \le |\Psi| + |b|,
$$ for all $b \in B(a, R)$  and this will be satisfied possibly on a $\partial B (0, r^{'})$ with $r^{'} \le r$ which is a contraddiction.\\
Hence we can assume that $\sup\limits_{|z| <1} (1-|z|) ||\Phi_n^{'} (z)|| =\lambda_n \to + \infty.$ \\
As in the previous result we can reparametrize and get
$$
\Psi_n := \Phi_n (a_n + A_n z) \to \Psi , \,\,\,\,   \Psi^{'} (0)=Id.
$$   
Then $\Psi(\mathbb{B}^k)$ contains a ball centered at $\Psi (0)=a,$ $B(\Psi(0), R).$ \\
Since by Rouche's theorem: 
$$
|\Psi_n - \Psi| < |\Psi|
$$
on $\partial B(0,r),$ then eventually the image of $\Psi_n$ contains a ball of radius smaller  or equal than $R$ because we are interested in the inequality 
$$|\Psi_n - b - \Psi +b| < |\Psi - b| \le |\Psi| + |b|,
$$ for all $b \in B(a, R)$  and this will be satisfied possibly on a $\partial B (0, r^{'})$ with $r^{'} \le r.$ \\
Then the image of $\Phi_n$ will contain eventually a ball of radius $R^{'} \le R$ and we get the domain $U.$  \\
The supremum of the set of positive numbers $R^{'}$ such that $\Phi(\mathbb{B}^k)$ contains a ball of radius $R^{'}$ is the so called {\it Landau's number} $l(\Phi).$ The constant $\rho$ is $\inf\limits_{\Phi \in \mathcal{G}} l(\Phi)$ where
$\mathcal{G}$ is the set of $\Phi \colon \mathbb{B}^k \to \mathbb{C}^k,$ $\Phi (0)=0,$ and $\Phi^{'} (0)= Id$ satisfying condition \eqref{DiffCond2}.
\end{proof}
We point out that, following the same arguments of \cite{Hahn73}, it is possible to find an open set $U$ on which the entire family of $\Phi$'s satisfying \eqref{DiffCond2} is injective, and it is also possible to relate the constant $C$ with the uniform radius $\rho.$
\begin{example}
Consider the following family of h\'enon maps of $\mathbb{C}^2:$
$$
h_b(z,w)=(z^2+bw, z)
$$
with $\delta < |b| \le  \eta,$ for $\delta, \,\, \eta$ fixed constants. \\
Let $f(z,w)=(e^{cz}-1,e^{cw}-1),$ with $|c|$ small enough if you need $f$ invertible, then
the family $$g_{b,c} (z,w) = \{ h_b \circ (e^{cz}-1, e^{cw}-1) \}$$ satisfies the hypothesis of theorem \ref{MainTheo}, up to dilation. \\
Furthermore, if the ball centered in the origin is enough small, then the ball of universal radius contained in the images of the family is a ball on which the maps of the family are invertible.
\end{example}
\begin{corollary}
Let $f\colon \mathbb{C}^k \to \mathbb{C}^k$ be an entire map. Suppose that $f^{'} (0) =Id$ and suppose that there exists a constant $C$ such that $||f^{'} (z)|| \cdot ||f^{'} (z)^{-1}|| \le C,$ then $f(\mathbb{C}^k)$ contains balls of arbitrary large radius.  
\end{corollary}
\begin{proof}
Indeed $\frac{1}{R}f(R \cdot B(0,1))\supset B(a,r)$ for any $R>0,$ by theorem \ref{MainTheo}.
\end{proof}

\section*{Acknowledgement}
The author warmly thanks the Referee for his/her accurate remarks.

\end{document}